\documentclass{amsart}
\usepackage[utf8]{inputenc}
\usepackage{amsmath, amssymb, amsthm}
\usepackage{mathrsfs}
\usepackage{enumerate}

\newtheorem{thm}{Theorem}[section]

\newtheorem{cor}[thm]{Corollary}
\newtheorem{lem}[thm]{Lemma}
\newtheorem*{claim*}{Claim}
\theoremstyle{definition}
\newtheorem{defi}[thm]{Definition}

\newtheorem{prob}[thm]{Problem}

\newcommand{\E}{\mathcal{E}}

\renewcommand{\O}{\mathcal{O}}
\renewcommand{\AE}{\mathcal{AE}}
\renewcommand{\P}{\mathcal{P}}
\newcommand{\R}{\mathcal{R}}

\newcommand{\AR}{\mathcal{AR}}

\renewcommand{\O}{\mathcal{O}}
\newcommand{\X}{\mathcal{X}}

\newcommand{\leqc}{\leq^{\circ}}

\newcommand{\U}{\mathcal{U}}
\newcommand{\V}{\mathcal{V}}
\newcommand{\W}{\mathcal{W}}
\newcommand{\I}{\mathcal{I}}
\DeclareMathOperator{\range}{range}
\DeclareMathOperator{\depth}{depth}

\DeclareMathOperator{\dom}{dom}

\begin{document}

\title[Topological Ramsey spaces of equivalence relations]{Topological Ramsey spaces of equivalence relations and a dual Ramsey theorem for countable ordinals}
\date{\today}

\author[J. K. Kawach]{Jamal K. Kawach}
\thanks{The first author is partially supported by an Ontario Graduate Scholarship.}
\address{Department of Mathematics\\ University of Toronto\\ Toronto, Canada, M5S 2E4.}
\email{jamal.kawach@mail.utoronto.ca}
\urladdr{https://www.math.toronto.edu/jkawach}

\author[S. Todorcevic]{Stevo Todorcevic}
\thanks{The second author is partially supported by grants from NSERC (455916) and CNRS (IMJ-PRGUMR7586).}
\address{Department of Mathematics\\ University of Toronto\\ Toronto, Canada, M5S 2E4.}
\address{Institut de Math\'ematique de Jussieu \\ UMR 7586, Case 247 \\ 4 place Jussieu, 75252 Paris Cedex, France}
\email{stevo@math.toronto.edu, todorcevic@math.jussieu.fr}


\begin{abstract}
We define a collection of topological Ramsey spaces consisting of equivalence relations on $\omega$ with the property that the minimal representatives of the equivalence classes alternate according to a fixed partition of $\omega$. To prove the associated pigeonhole principles, we make use of the left-variable Hales-Jewett theorem and its extension to an infinite alphabet. We also show how to transfer the corresponding infinite-dimensional Ramsey results to equivalence relations on countable limit ordinals (up to a necessary restriction on the set of minimal representatives of the equivalence classes) in order to obtain a dual Ramsey theorem for such ordinals.
\end{abstract}

\maketitle

\section{Introduction}
Recall that the infinite Ramsey theorem \cite{R} says that for every $k < \omega$ and every finite colouring of the collection of $k$-element subsets of $\omega$, there is an infinite $X \subseteq \omega$ such that the colouring is constant on the $k$-element subsets of $X$. To obtain an analogous result for finite colourings of the collection $[\omega]^\omega$ of all infinite subsets of $\omega$, one needs to impose a topological restriction on the fibres of the colouring, where $[\omega]^\omega$ is equipped with its standard metrizable topology. Galvin and Prikry \cite{GP} showed that Borel colourings of $[\omega]^\omega$ have the desired property; Silver \cite{S} then extended this to analytic colourings of $[\omega]^\omega$.

In \cite{CS1} Carlson and Simpson prove the following ``dual'' version of the infinite Ramsey theorem, which involves colouring partitions of $\omega$ instead of subsets and which can be seen as an infinite version of the finite dual Ramsey theorem of Graham and Rothschild \cite{GR}. To state it, we establish some notation: Given $k <  \omega$, let $\E_k$ be the set of all equivalence relations on $\omega$ with exactly $k$-many equivalence classes. We also let $\E_\infty$ be the set of all equivalence relations on $\omega$ with infinitely many equivalence classes. Given $X \in \E_\infty$ and $\alpha \in \omega \cup \{\infty\}$, we write $\E_\alpha \restriction X$ for the set of all $Y \in \E_\alpha$ which are \emph{coarsenings} of $X$ in the sense that each equivalence class of $Y$ is a union of equivalence classes of $X$. Carlson and Simpson proved the following result for Borel colourings, but we state it in its stronger form as found in \cite[Chapter 5.6]{T}.
 
\begin{thm}[Dual Ramsey theorem]
For every finite Souslin measurable colouring $c$ of $\E_k$, there is $X \in \E_\infty$ such that $c$ is constant on $\E_k \restriction X$.
\end{thm}

The topology on the space of equivalence relations and the notion of Souslin measurability will be defined in the next few sections.

Carlson and Simpson also proved a dual version of the Galvin-Prikry theorem. As above, we state the following result in its optimal form from \cite{T}. 

\begin{thm}[Dual Silver theorem]
For every finite Souslin measurable colouring $c$ of $\E_\infty$, there is $X \in \E_\infty$ such that $c$ is constant on $\E_\infty \restriction X$.
\end{thm} 

The main goal of this paper is to obtain similar results for any countable limit ordinal in place of $\omega$, up to a restriction on the placement of the minimal representatives of each equivalence relation. To do so, we use topological Ramsey spaces and code equivalence relations on a countable limit ordinal via certain ``alternating'' equivalence relations on $\omega$. This general machinery is then used to obtain the following versions of the dual Silver theorem, which correspond to Corollary \ref{maincor1} and Corollary \ref{maincor2}, respectively. (The definitions of the spaces of equivalence relations and their associated topologies will be given later.)

\begin{thm}
For every $l \geq 1$ there is an optimal collection $\E_\infty(\omega \cdot l)$ of equivalence relations on $\omega \cdot l$ with infinitely many equivalence classes such that the following holds: For every finite Souslin measurable colouring of $\E_\infty(\omega \cdot l)$ and every $E \in \E_\infty(\omega \cdot l)$, there is a coarsening $F \in \E_\infty(\omega \cdot l)$ of $E$ such that the family $\E_\infty(\omega \cdot l)\restriction F$ of all coarsenings of $F$ is monochromatic.
\end{thm}

\begin{thm}
Let $\alpha \geq \omega^2$ be a countable limit ordinal and suppose $\beta$ is an ordinal such that $\alpha = \omega \cdot \beta$. Then for every bijection $f : \omega \rightarrow \beta$ there is an optimal collection $\E^f_\infty(\alpha)$ of equivalence relations on $\alpha$ with infinitely many equivalence classes such that the following holds: For every finite Souslin measurable colouring of $\E^f_\infty(\alpha)$ and every $E \in \E^f_\infty(\alpha)$, there is a coarsening $F \in \E^f_\infty(\alpha)$ of $E$ such that the family $\E^f_\infty(\alpha)\restriction F$ of all coarsenings of $F$ is monochromatic.
\end{thm}

This paper is organized as follows: Section 2 contains a brief overview of the theory of topological Ramsey spaces together with the statement of the left-variable Hales-Jewett theorem, which is the main pigeonhole principle that will be needed later. In Section 3, we extend the dual Ramsey and dual Silver theorems to the ordinals $\omega \cdot l$ by considering topological Ramsey spaces of equivalence relations on $\omega$ which alternate (in a sense that will be made precise later) with respect to a finite partition of $\omega$. In Section 4, we obtain a version of the left-variable Hales-Jewett theorem for infinite alphabets. We then use this pigeonhole principle in Section 5 to prove dual Ramsey and dual Silver theorems for any countable limit ordinal $\alpha \geq \omega^2$ by developing topological Ramsey spaces of equivalence relations which alternate over an infinite partition of $\omega$.

\subsection*{Acknowledgements} We are grateful to Jordi L\'opez-Abad for many useful conversations on the subject matter and for comments on an earlier version of this paper. We also thank the referee for many helpful comments and suggestions.

\section{Preliminaries}

\subsection{Topological Ramsey spaces}

The exposition in this subsection closely follows that of \cite[Chapter 5]{T} and we refer the reader there for more information. Let $\R$ be a set which we will think of as a collection of infinite sequences of objects. We equip $\R$ with a quasiorder $\leq$ and a \emph{restriction map} $$r : \R \times \omega \rightarrow \AR$$ where $$\AR = \bigcup_{n <\omega} \AR_n$$ and where $\AR_n$ (which we think of as the collection of $n^{\rm{th}}$ approximations to elements of $\R$) is the range of $r_n = r\restriction{\R \times \{n\}}$. We will be interested in structures $$(\R,\leq, r)$$ of the above form which moreover satisfy the following four axioms. The relevant notation will be explained after the statement of the axioms.

\begin{itemize}
	\item[\textbf{A.1.}] For all $A, B \in \R$ we have:
		\begin{enumerate}
			\item $r_0(A) = \emptyset$.
			\item $A \neq B$ implies $r_n(A) \neq r_n(B)$ for some $n$.
			\item $r_n(A) = r_m(B)$ implies $n = m$ and $r_k(A) = r_k(B)$ for all $k < n$.
		\end{enumerate}
	\item[\textbf{A.2.}] There is a quasiorder $\leq_{\text{fin}}$ on $\AR$ such that:
		\begin{enumerate}
			\item $\{a \in \AR: a \leq_{\text{fin}} b\}$ is finite for all $b \in \AR$.
			\item For all $A, B \in \R$ we have $$A \leq B \text{ iff } (\forall n) (\exists m) \, r_n(A) \leq_{\text{fin}} r_m(B).$$
			\item For all $a, b \in \AR$, $$a \sqsubseteq b \wedge b \leq_{\text{fin}} c \rightarrow (\exists d \sqsubseteq c) \, a \leq_{\text{fin}} d.$$
		\end{enumerate}
	\item[\textbf{A.3.}] Let $A, B \in \R$ and $a \in \AR$ be such that $d = \depth_B(a) < \infty$. Then:
		\begin{enumerate}
		\item $[a,A'] \neq \emptyset$ for all $A' \in [r_d(B),B]$.
		\item If $A \leq B$ and $[a,A] \neq \emptyset$, then there is $A' \in [r_d(B), B]$ such that $[a,A'] \subseteq [a,A]$.
		\end{enumerate}
	\item[\textbf{A.4.}] If $[a,B] \neq \emptyset$ and $d = \depth_B(a) < \infty$, then for every $\O \subseteq \AR_{|a|+1}$ there exists $A \in [r_d(B), B]$ such that $r_{|a|+1}[a, A] \subseteq \O$ or $r_{|a|+1}[a, A] \subseteq \O^c$.
\end{itemize}

Above, we use the notation $|a|$ to denote the \emph{length} of $a \in \AR$, which is the unique integer $n$ such that $a = r_n(A)$ for some $A \in \R$. We also write $a \sqsubseteq b$ whenever there is $B \in \R$ and $n \leq m < \omega$ such that $a = r_n(B)$ and $b = r_m(B)$. For $a \in \AR$ and $B \in \R$, we define $$[a, B] = \{A \in \R : A \leq B \text{ and } (\exists n) \,  r_n(A) = a\}.$$ Also recall that the \emph{depth} of $a$ in $B$, for $a \in \AR$ and $B \in \R$, is defined by
\[ \depth_B(a) = \begin{cases}
	\min\{n < \omega : a \leqc_{\text{fin}} r_n(B)\} & \text{ if $(\exists n) \, a \leqc_{\text{fin}} r_n(B)$},\\
	\infty & \text{ otherwise}.
	\end{cases}
\]

We equip $\AR$ with the discrete topology, $\AR^\omega$ with the corresponding product topology, and $\R$ with the subspace topology from $\AR^\omega$. We say $(\R,\leq,r)$ is \emph{closed} if $\R$ is closed when viewed as a subspace of $\AR^\omega$. The \emph{Ellentuck topology} on $\R$ is the topology generated by basic open sets of the form $[a, B]$ for $a \in \AR, B \in \R$. Note that this topology refines the metrizable topology on $\R$ when considered as a subspace of $\AR^\omega$.

\begin{defi} Let $\X \subseteq \R$.
\begin{enumerate}[(i)]
\item $\X$ has the \emph{property of Baire} if $\X = \O \triangle M$ for some Ellentuck open $\O\subseteq \R$ and some Ellentuck meagre $M \subseteq \R$.
\item $\X$ is \emph{Ramsey} if for every non-empty basic set $[a, B]$ there is $A \in [a,B]$ such that $[a,A] \subseteq \X$ or $[a,A] \cap \X = \emptyset$. If the second alternative always holds, then $\X$ is \emph{Ramsey null}.
\end{enumerate}
\end{defi}

\begin{thm}[Abstract Ellentuck Theorem]
Suppose $(\R, \leq, r)$ is closed and satisfies axioms A.1 to A.4. Then $\X \subseteq \R$ has the property of Baire if and only if it is Ramsey. Furthermore, $\X \subseteq \R$ is Ellentuck meagre if and only if it is Ramsey null.
\end{thm}

A structure of the above form which satisfies the conclusion of the Abstract Ellentuck Theorem is called a \emph{topological Ramsey space}.

Recall that a \emph{Souslin scheme} is a family of subsets $(X_s)_{s \in \omega^{< \omega}}$ of some underlying set which is indexed by finite sequences of non-negative integers. The \emph{Souslin operation} turns a Souslin scheme $(X_s)_{s \in \omega^{< \omega}}$ into the set $$\bigcup_{x \in \mathcal{N}} \bigcap_{n < \omega} X_{x \restriction n}$$ where $\mathcal{N}$ denotes the Baire space, i.e. the set of all infinite sequences in $\omega$. Then $\X \subseteq \R$ is \emph{Souslin measurable} if it belongs to the minimal field of subsets of $\R$ which contains all Ellentuck open sets and is closed under the Souslin operation. In particular, every analytic or coanalytic subset of $\R$ is Souslin measurable. Finally, we say finite colouring $c : \R \rightarrow n$ is Souslin measurable if each set $c^{-1}\{i\}, \, i < n$ is Souslin measurable.

We make note of the following useful consequence of the Abstract Ellentuck Theorem. From now on, all topological notions will refer to the Ellentuck topology on $\R$.

\begin{cor}
Suppose $(\R, \leq, r)$ is closed and satisfies axioms A.1 to A.4. Then for every finite Souslin measurable colouring of $\R$ and every $B \in \R$, there is $A \leq B, A \in \R$ such that the set $\{A' \in \R : A' \leq A\}$ is monochromatic.
\end{cor}

\subsection{The left-variable Hales-Jewett theorem}

The \emph{Hales-Jewett theorem}, originally proved in \cite{HJ}, is a powerful combinatorial result which involves colouring finite words over an alphabet. We will be concerned with a particular infinitary version of the Hales-Jewett theorem which we now proceed to describe. Let $L$ be a finite set which we refer to as an \emph{alphabet}, and let $v$ be a symbol distinct from all members of $L$. We let $W_L$ denote the set of all \emph{words} formed from $L$, i.e. all functions $w : n \rightarrow L$ where $n < \omega$. Similarly, we let $W_{Lv}$ denote all \emph{variable words}, i.e. functions $x : n \rightarrow L \cup \{v\}$ where $n < \omega$ and $v \in \range(x)$. We will often write $|w|$ for the domain (or length) of the word $w$. For two words $x$ and $y$ we write $x^\frown y$ for the word obtained by concatenating $x$ and $y$. If $x$ is a variable word and $\lambda \in L \cup \{v\}$, we write $x[\lambda]$ for the element of $W_L$ or $W_{Lv}$ obtained by replacing every occurrence of the variable $v$ in $x$ by $\lambda$. Given a sequence $X = (x_n)_{n<\omega}$ of variable words over an alphabet $L$, we write $[X]_L$ for the \emph{partial subsemigroup of $W_L$ generated by $X$}, which consists of all words of the form $${x_{n_0}[\lambda_0]}^\frown \dots ^\frown{x_{n_k}}[\lambda_k]$$ where $k < \omega, n_0 < \dots < n_k < \omega$ and $\lambda_i \in L$ for each $i \leq k$. The partial subsemigroup $[X]_{Lv}$ of $W_{Lv}$ is defined similarly.

We will make use of the following version of the infinite Hales-Jewett theorem originally considered in \cite{CS1} and obtained in the form below in \cite{HM}; another proof can be found in \cite[Chapter 2]{T}. First, let us say that $x \in W_{Lv}$ is a \emph{left-variable word} if the first letter of $x$ is $v$.

\begin{thm}[Left-variable Hales-Jewett Theorem]
Suppose $L$ is a finite alphabet. Then for every finite colouring of $W_L$ there is a sequence $X = (x_n)_{n<\omega}$ of left-variable words together with a variable-free word $w_0$ such that the translate ${w_0}^\frown [X]_L$ is monochromatic.
\end{thm}

\section{$\P$-alternating equivalence relations for finite partitions of $\omega$}

Let $\P$ be a partition of $\omega$ into finitely many infinite sets $P_0, \dots, P_{l-1}$ such that $\min P_i < \min P_j$ whenever $i < j$. Let $\E^\P_\infty$ denote the set of all \emph{$\P$-alternating} equivalence relations on $\omega$, i.e. equivalence relations $E$ on $\omega$ such that:
\begin{enumerate}
	\item[(a)] $E$ has infinitely many equivalence classes.
	\item[(b)] If $(p_k(E))_{k<\omega}$ is an increasing enumeration of the set of minimal representatives of the equivalence classes of $E$, then $p_k(E) \in P_n$ where $n \equiv  k \mod l$.
\end{enumerate}
Given $E \in \E^\P_\infty$, a \emph{$P_n$-class} is an equivalence class $X$ of $E$ such that $\min X \in P_n$. Note that condition (b) implies that each $E \in \E^\P_\infty$ has infinitely many $P_n$-classes for each $n < l$.

Before proceeding, we establish some notation and terminology which will be used throughout the paper when working with equivalence relations. For any two equivalence relations $E$ and $F$ on the same set, write $E \leq F$ whenever every class of $E$ is a union of a classes of $F$. In this case we also say that $E$ is a \emph{coarsening} of $F$, or that $E$ is \emph{coarser} than $F$. For any equivalence relation $E$ on a well-ordered set, $p(E)$ will denote the set of all minimal representatives of the classes of $E$. The \emph{$n^{\rm{th}}$ approximation} of an equivalence relation $E$ on $\omega$ is given by $$r_n(E) = E \restriction p_n(E)$$ where $p_n(E)$ is the $n^{\mathrm{th}}$ element of $p(E)$ when the latter set is enumerated in increasing order. Let $\AE^\P_\infty$ denote the set of all finite approximations of elements of $\E^\P_\infty$. For $a \in \AE^\P_\infty$, let $|a|$ denote the \emph{length} of $a$, which is defined as the unique integer such that $a = r_n(E)$ for some $E$; equivalently, $|a|$ is the number of equivalence classes of $a$. $(\AE^\P_\infty)_n$ will denote all approximations of length $n$. The \emph{domain} of an approximation $a$ is the set $$\dom(a) = \{0, 1, \dots, p_{|a|}(E) - 1\} = p_{|a|}(E)$$ where $E$ is such that $a = r_n(E)$. The relation $\leq$ admits a finitization obtained by setting $a \leq_\text{fin} b$ if and only if $\dom(a) = \dom(b)$ and $a$ is coarser than $b$.

\begin{thm}\label{thm1}
$(\E^\P_\infty, \leq, r)$ is a topological Ramsey space.
\end{thm}

\begin{proof}
As in \cite[Chapter 5.6]{T}, it is routine to check that $(\E^\P_\infty, \leq, r)$ is closed and satisfies the Ramsey space axioms A.1, A.2 and A.3. Hence it is enough to prove A.4; we follow the proof of \cite[Lemma 5.69]{T}. So let $[a,E] \neq \emptyset$ be a basic set, $n = |a|$ and $\O \subseteq (\AE^\P_\infty)_{n+1}$. Note that we may assume $[a,E]$ has the property that $a = r_n(E)$, since then the arbitrary case follows from A.3(2) by taking an appropriate coarsening of $a$. Assuming $a = r_n(E)$, it follows that $\depth_E(a) = n$ and so we want to find $F \in [a, E]$ such that $r_{n+1}[a,F] \subseteq \O$ or $r_{n+1}[a,F] \subseteq \O^c$.

Consider an arbitrary end-extension $b \in r_{n+1}[a,E]$. Then $b$ is an equivalence relation on a set of the form $$p_m(E) = \{0, 1, \dots, p_m(E) - 1\}$$ for some $m > n$, such that $b$ has one more equivalence class with minimal representative $p_n(E)$. Thus $b$ joins the classes of $E$ with minimal representatives among $p_{n+1}(E), \dots, p_{m-1}(E)$ to a class with minimal representative $\leq p_n(E)$. Note that if $n \equiv k \mod l$ then $m \equiv k+1 \mod l$ because of our condition for being a member of $\E^\P_\infty$. Let $\lambda := \frac{m-n-1}{l}$. Then any such $b$ can be coded as a word $w^b$ in the alphabet $$L = (n+1)^l$$ such that $w^b$ has length $\lambda$: If we let $\pi_j(w^b(i))$ denote the $j^{\mathrm{th}}$ coordinate of the letter $w^b(i)$, then for each $i < \lambda$ and $j < l$, we set $\pi_j(w^b(i)) = k$ where $p_{(n+1)+il+j}(E)$ is joined to $p_k(E)$ via $b$. In other words, each block of the form $$\left(p_{n+il + 1}(E), p_{n+il+2}(E), \dots, p_{n+(i+1)l}(E)\right), \, \, \,  i < \lambda$$ is associated to the letter $$(k_0, \dots, k_{l-1}) \in L$$ where $p_{k_j}(E)$ is joined to $p_{(n+1)+il+j}(E)$. Conversely, any word $w \in W_L$ which has length $\lambda'$ corresponds to a unique $b = b(w) \in r_{n+1}[a,E]$ where $\dom(b) = r_{n + 1 + l\cdot \lambda'}(E)$: For each $i < \lambda'$ and $j < l$, join the class $p_{(n+1)+il+j}(E)$ to $p_{\pi_j(w(i))}(E)$ and let $b(w)$ be the corresponding equivalence relation on the set $$p_m(E) =  \{0, 1, \dots, p_m(E) - 1\}.$$

Define a colouring $c : W_L \rightarrow 2$ by setting $c(w) = 0$ if and only if $b(w) \in \O$, and apply the left-variable Hales-Jewett theorem to obtain a variable-free word $w_0$ together with a sequence $X = (x_i)_{i<\omega}$ of left-variable words such that the translate ${w_0}^\frown[X]_L$ is monochromatic for the above colouring. In other words, either
\begin{enumerate}
\item $b(w) \in \O$ for every $w \in {w_0}^\frown[X]_L$, or
\item $b(w) \not \in \O$ for every $w \in {w_0}^\frown[X]_L$.
\end{enumerate}
For each $i<\omega$, the left-variable word $x_i$ determines a variable word $y_i$ of length $l\cdot |x_i|$ in the alphabet $\{0, \dots, n\}$ which is obtained by replacing each letter $(k_1, \dots, k_l) \in L$ with the string $$k_1 \dots k_l$$ and replacing the variable $v$ with the string $$0 \dots 0 v 0 \dots  0$$ of length $l$ where $v$ occurs in the $(i \mod l)^{\mathrm{th}}$ place. Similarly, $w_0$ corresponds to a word $u_0$ of length $l \cdot |w_0|$ in the alphabet $\{0,\dots, n\}$. Now form the infinite word $$y = {u_0}^\frown{y_0}^\frown{y_1}^\frown \dots ^\frown{y_k}^\frown \dots$$ out of $u_0$ and $(y_i)_{i<\omega}$. To define $F$, it suffices to say how it acts on $p(E)$. If at place $i$ of $y$ we find a letter $k \in \{0,\dots,n\}$, we let $p_{n+1+i}(E)$ and $p_k(E)$ be $F$-equivalent. If we find a variable at places $$i, i' \in [|u_0| + |y_0| + \dots + |y_{j-1}|, |u_0| + |y_0| + \dots + |y_{j-1}| + |y_j|)$$ (where we set $y_{-1} = \emptyset$) then we let $p_{n+1+i}(E)$ and $p_{n+1+i'}(E)$ be $F$-equivalent, with no other connections. In this way, the sequence of minimal representatives $(p_i(F))_{i<\omega}$ of $F$ is a subsequence of $(p_i(E))_{i<\omega}$ obtained by removing blocks which have length a multiple of $l$, and so $F$ is $\P$-alternating.

It remains to check that $F$ satisfies the conclusion of axiom A.4. It is clear that $F \in [a, E]$. Now take an arbitrary end-extension $b \in r_{n+1}[a,F]$. By our condition for being a member of $\E^\P_\infty$, $b$ must be an equivalence relation on a set of the form $p_m(F)$ where $m \equiv n+1 \mod l$. In particular, this implies that $b$ can be obtained from a word of the form $${u_0}^\frown{y_0[\lambda_0]}^\frown \dots ^\frown{y_k[\lambda_k]}$$ for some $k < \omega$ and $\lambda_0, \dots, \lambda_k \in \{0,\dots, n\}$, and so $b$ corresponds to a word $${w_0}^\frown{x_0[\lambda'_0]}^\frown \dots ^\frown{x_k[\lambda'_k]}$$ for some $\lambda'_0, \dots, \lambda'_k \in L$. Thus $F$ satisfies the conclusion of A.4.
\end{proof}

\begin{cor}\label{cor1}
Suppose $c$ is a finite Souslin measurable colouring of $\E^\P_\infty$. Then for every $E \in \E^\P_\infty$ there is $F \leq E$ such that the family $\E^\P_\infty\restriction F$ of all coarsenings of $F$ is $c$-monochromatic.
\end{cor}

In the above space we only prescribed conditions on the minimal representatives associated to an equivalence relation. We can also prescribe conditions on \emph{all} of elements belonging to an equivalence class rather than just the minimal representatives. One relevant instance of this will be the following: Let $\P$ be a partition of $\omega$ into finitely many sets $P_0, \dots, P_{l-1}$ such that $\min P_i < \min P_j$ whenever $i < j$ and let $\E^{\P'}_\infty$ denote the set of all $\P$-alternating equivalence relations $E$ on $\omega$ which also satisfy:
\begin{enumerate}
	\item[(c)] If $X$ is a $P_n$-class, then $X \subseteq \bigcup_{m\geq n} P_m$.
\end{enumerate}

\begin{thm}\label{thm2}
$(\E^{\P'}_\infty, \leq, r)$ is a topological Ramsey space.
\end{thm}

\begin{proof}
As in the proof of Theorem \ref{thm1}, it is enough to check axiom A.4. So we let $[a, E] \neq \emptyset$ be a basic set, $n = |a|$ and $\O \subseteq (\AE^{\P'}_\infty)_{n+1}$. As before, we assume $a = r_n(E)$ so and we aim to find $F \in [a,E]$ such that $r_{n+1}[a,F] \subseteq \O$ or $r_{n+1}[a,F] \subseteq \O^c$. Consider an arbitrary end-extension $b \in r_{n+1}[a,E]$. Then $b$ is an equivalence relation on a set of the form $$p_m(E) = \{0, 1, \dots, p_m(E) - 1\}$$ for some $m > n$, such that $b$ has one more equivalence class with minimal representative $p_n(E)$. Thus $b$ joins the classes of $E$ with minimal representatives among $p_{n+1}(E), \dots, p_{m-1}(E)$ to a class with minimal representative $\leq p_n(E)$. Note that if $n \equiv k \mod l$ then $m \equiv k+1 \mod l$ because of condition (b) in the definition of $\P$-alternating. Furthermore, condition (c) implies that each $P_i$-class can only be joined to an earlier $P_j$-class where $j \leq i$.

Exactly as in the proof of Theorem \ref{thm1}, any end-extension $b$ as above can be coded as a word $w^b$ in the alphabet $$L = (n+1)^l$$ such that $w^b$ has length $\lambda := \frac{m-n-1}{l}$. Conversely, any word $w \in W_L$ which has length $\lambda'$ corresponds to a unique $b = b(w) \in r_{n+1}[a,E]$ where $\dom(b) = r_{n + 1 + l\cdot \lambda'}(E)$. To describe this assignment, we first define for each $w \in W_L \cup W_{Lv}$ a word $\widetilde{w}$ of the same length as follows: If the $i^{\mathrm{th}}$ position of $w$ is occupied by the variable $v$, then we let $\widetilde{w}(i) = v$. Otherwise, $w(i) \in L$ and we define $\widetilde{w}(i)$ by setting the $j^{\mathrm{th}}$ coordinate of $\widetilde{w}(i)$ to be
\[ \pi_j(\widetilde{w}(i)) = \begin{cases}
	\pi_j(w(i)) & \text{ if $p_{(n+1) + il + j}(E)$ can be joined to $p_{\pi_j(w(i))}(E)$},\\
	0 & \text{ otherwise}
	\end{cases}
\]
for each $j < l$. Then, given $w \in W_L$ as above, for each $i < \lambda'$ and $j < l$ join the class with minimal representative $p_{(n+1)+il+j}(E)$ to $p_{\pi_j(\widetilde{w}(i))}(E)$ and let $b(w)$ be the corresponding equivalence relation on the set $$p_m(E) =  \{0, 1, \dots, p_m(E) - 1\}.$$

Define a colouring $c : W_L \rightarrow 2$ by setting $c(w) = 0$ if and only if $b(w) \in \O$, and apply the left-variable Hales-Jewett theorem to obtain a variable-free word $w_0$ together with a sequence $X = (x_n)_{n<\omega}$ of left-variable words such that either
\begin{enumerate}
\item $b(w) \in \O$ for every $w \in {w_0}^\frown[X]_L$, or
\item $b(w) \not \in \O$ for every $w \in {w_0}^\frown[X]_L$.
\end{enumerate}
Using $w_0$ and $X$, construct (exactly as in the proof of Theorem \ref{thm1}) a variable-free word $u_0$ and a sequence of left-variable words $(y_i)_{i<\omega}$ over the alphabet $\{0, \dots, n\}$ and then form the infinite variable word $$y = {u_0}^\frown{y_0}^\frown{y_1}^\frown \dots ^\frown{y_k}^\frown \dots$$ over the alphabet $\{0, \dots, n\}$. Define an infinite variable word $\widetilde{y}$ by setting the $i^{\mathrm{th}}$ letter to be
\[ \widetilde{y}(i) = \begin{cases}
	y(i) & \text{ if $p_{n + 1 + i}(E)$ can be joined to $p_{y(i)}(E)$},\\
	v & \text{ if $y(i) = v$},\\
	0 & \text{ otherwise}.
	\end{cases}
\]
and use it to define an equivalence relation $F$ exactly as in the proof of Theorem \ref{thm1}. The placement of the variables in the infinite word $\widetilde{y}$ (which are the same as that of $y$) ensures that $F$ is $\P$-alternating, while condition (c) in the definition of $\E^{\P'}_\infty$ is satisfied by the fact that we replaced $y$ with $\widetilde{y}$. Then each end-extension $b \in r_{n+1}[a,F]$ can be obtained as $b = b(w)$ for some word $$w = {w_0}^\frown{x_0[\lambda'_0]}^\frown \dots ^\frown{x_k[\lambda'_k]}$$ where $\lambda'_0, \dots, \lambda'_k \in L$. It then follows from our choice of $w_0$ and $X$ that $F$ satisfies the conclusion of A.4.
\end{proof}

\begin{cor}\label{cor2}
Suppose $c$ is a finite Souslin measurable colouring of $\E^{\P'}_\infty$. Then for every $E \in \E^{\P'}_\infty$ there is $F \leq E$ such that the family $\E^{\P'}_\infty\restriction F$ of all coarsenings of $F$ is $c$-monochromatic.
\end{cor}

To conclude this section, we apply Corollary \ref{cor2} to prove an analogue of the dual Ramsey theorem for ordinals of the form $\omega \cdot l$. In this setting it is natural to work with equivalence relations whose equivalence classes are precisely the fibres of a \emph{rigid surjection} $\omega \cdot l \rightarrow \omega \cdot l$, i.e. a surjection $f: \omega \cdot l \rightarrow \omega \cdot l$ such that $$\min f^{-1}(\alpha) < \min f^{-1}(\beta) \text{ for all $\alpha < \beta < \omega \cdot l$}.$$ When $l = 1$ there is a natural bijective correspondence between rigid surjections and equivalence relations, but for $l > 1$ this is no longer the case; in particular, it is easy to check that a rigid surjection $f : \omega \cdot l \rightarrow \omega \cdot l$ corresponds uniquely to an equivalence relation $E$ on $\omega \cdot l$ such that $p(E)$ has infinite intersection with each copy of $\omega$. Thus, we will work exclusively with equivalence relations on $\omega \cdot l$ which arise from rigid surjections.

Working with any partition $\P$ of $\omega$ as above, we view each $P_i, \, i < l$ as the $i^{\mathrm{th}}$ copy of $\omega$ in $\omega \cdot l$, i.e. $P_i$ is identified with $\omega \times \{i\}$. In order to obtain a Ramsey theorem in this setting, we need to prescribe conditions on the set of minimal representatives of an equivalence relation on $\omega \cdot l$; in particular we need to ensure that there is only one possible ``pattern'' for the set of minimal representatives of an equivalence relation. A natural requirement is to ask that these representatives cycle between the different copies of $\omega$ when ordered according to the lexicographical ordering $\preceq$ on $\omega \cdot l$, i.e. $(n, i) \preceq (m, j)$ if and only if $n < m$, or $n = m$ and $i \leq j$. Additionally, it is necessary to ask that the minimal representatives (under the standard ordering on $\omega \cdot l$) agree with the $\preceq$-minimal representatives. Thus, we let $\E_\infty(\omega \cdot l)$ be the set of all equivalence relations $E$ on $\omega \cdot l$ such that the following conditions hold:
	\begin{enumerate}[(a)]
		\item If $(p_k)_{k<\omega}$ is a $\preceq$-increasing enumeration of $p(E)$, then $$p_k \in \omega \times \{ k \mod l\}$$ for each $k$.
		\item If $(q_k)_{k<\omega}$ is a $\preceq$-increasing enumeration of the set of $\preceq$-minimal representatives of $E$, then $q_k = p_k$ for all $k$.
	\end{enumerate}
	
In order to transfer equivalence relations from $\omega$ to $\omega \cdot l$ (and vice versa) we first define a mapping $\phi : \omega \rightarrow \omega \cdot l$ as follows: Given $n < \omega$, let $\theta(n)$ be the unique integer such that $n \in P_{\theta(n)}$, and define a mapping $\psi : \omega \rightarrow \omega$ by setting $\psi(n) = m$ if and only if $n$ is the $m^{\mathrm{th}}$ element of $P_{\theta(n)}$. Then let $$\phi(n) = (\psi(n), \theta(n))$$ and note that $\phi$ is a bijection between $\omega$ and $\omega \cdot l$. Given $E = (E_n)_{n < \omega} \in \E^{\P'}_\infty$, define an equivalence relation on $\omega \cdot l$ by setting $$\Phi(E) = (\phi''E_n)_{n < \omega}.$$ The following lemma summarizes the main properties of $\Phi$ that we will need.

\begin{lem}
$\Phi$ is a bijective mapping between $\E^{\P'}_\infty$ and $\E_\infty(\omega \cdot l)$ which has the property that $F \leq E$ if and only if $\Phi(F) \leq \Phi(E)$.
\end{lem}
\begin{proof}
The second part of the lemma is immediate from the fact that $\phi$ is a bijection, while the first part will follow once we show that the image of $\Phi$ is contained in $\E_\infty(\omega \cdot l)$. To prove the latter it is enough to show that $\Phi$ preserves the set of minimal representatives of each equivalence relation $E \in \E^{\P'}_\infty$ in the sense that $\phi(p(E)) = p(\Phi(E))$. To this end, fix $E = (E_n)_{n < \omega} \in \E^{\P'}_\infty$. Then, for each $n < \omega$, $p_n = \min E_n$ belongs to $P_{\sigma(n)}$ where $\sigma(n) < l$ and $\sigma(n) \equiv n \mod l$, and so the definition of $\phi$ implies that $\phi(p_n) \in \omega \times \{\sigma(n)\}$. Since $E_n \subseteq \bigcup_{m \geq \sigma(n)} P_m$ it follows that for each $x \in E_n$ we have $\phi(x) \in \omega \times \{m\}$ for some $m \geq \sigma(n)$. If $m > \sigma(n)$ then $\phi(x) > \phi(p_n)$ in the usual ordering on $\omega \cdot l$, so assume $m = \sigma(n)$. By definition of $p_n$, we know $p_n \leq x$ in the ordering on $\omega$ and so $\psi(p_n) \leq \psi(x)$. Thus $\phi(p_n) \leq \phi(x)$ for any $x \in E_n$. Hence $\phi(p_n) = \min \phi'' E_n$ and so $$p(\Phi(E)) = \{\phi(p_n) : n < \omega\} = \phi(p(E)).$$ Since $\phi(p_n) \in \omega \times \{\sigma(n)\}$ for each $n$, this implies $\Phi(E) \in \E_\infty(\omega \cdot l)$. Thus the image of $\E^{\P'}_\infty$ under $\Phi$ is contained in $\E_\infty(\omega \cdot l)$. The fact that $\Phi: \E^{\P'}_\infty \rightarrow \E_\infty(\omega \cdot l)$ is a bijection is now straightforward.
\end{proof}

We can equip $\E_\infty(\omega \cdot l)$ with the topology inherited from $\E^{\P'}_\infty$ via $\Phi$, where the latter set is equipped with either the Ellentuck topology or the standard metrizable topology. Note that both of these topologies refine the topology induced from $2^{(\omega \cdot l)^2}$ when each equivalence relation in $\E_\infty(\omega \cdot l)$ is identified with a subset of $(\omega \cdot l)^2$ in the standard way, e.g. as in \cite{CS1}. Thus, the notion of Souslin measurability in the next result can refer to either of these three topologies. The following is the extension of the dual Silver theorem to ordinals of the form $\omega \cdot l$.

\begin{cor}\label{maincor1}
Suppose $c$ is a finite Souslin measurable $n$-colouring of $\E_\infty(\omega \cdot l)$. Then for every $E \in \E_\infty(\omega \cdot l)$ there is $F \leq E$ such that the family $\E_\infty(\omega \cdot l)\restriction F$ of all coarsenings of $F$ is $c$-monochromatic.
\end{cor}
\begin{proof}
Fix $c$ and $E$ as in the statement of the corollary. Let $\widetilde{c} = c \circ \Phi$; then $\widetilde{c}$ is Souslin measurable and so by Corollary \ref{cor2} there are $F \leq \Phi^{-1}(E)$ and $i < n$ such that $\widetilde{c}(F') = i$ for each $F' \in \E^{\P'}_\infty$ such that $F' \leq F$. Then $\Phi(F) \leq E$, and if $F' \in \E_\infty(\omega \cdot l)$ has the property that $F' \leq \Phi(F)$, then $$c(F') = c(\Phi(\Phi^{-1}(F'))) = \widetilde{c}(\Phi^{-1}(F')) = i$$ since $\Phi^{-1}(F') \leq F$. Thus $\Phi(F)$ satisfies the conclusion of the corollary. 
\end{proof}

We conclude this section with a description of the corresponding version of the dual Ramsey theorem for $\omega \cdot l$. For each $k < \omega$, let $\E_k(\omega \cdot l)$ denote the set of all equivalence relations $E$ on $\omega \cdot l$ with exactly $k$ equivalence classes such that if $(p_i)_{i<k}$ is a $\preceq$-increasing enumeration of $p(E)$ then $$p_i \in \omega \times \{ i \mod l\} \text{ for all $i < k$},$$ and if $(q_i)_{i<k}$ is a $\preceq$-increasing enumeration of the set of $\preceq$-minimal representatives of $E$, then $q_i = p_i$ for all $i< k$. Equip $\E_k(\omega \cdot l)$ with the topology inherited from $2^{(\omega \cdot l)^2}$ by identifying each equivalence relation in $\E_k(\omega \cdot l)$ with a subset of $(\omega \cdot l)^2$. We then have the following extension of the dual Ramsey theorem to $\omega \cdot l$. The proof is the same as the corresponding result from \cite[Corollary 5.72]{T}; we include it here for the sake of completeness.

\begin{cor}
Suppose $c$ is a finite Souslin measurable colouring of $\E_k(\omega \cdot l)$. Then for every $E \in \E_\infty(\omega \cdot l)$ there is $F \leq E$ such that the family $\E_k(\omega \cdot l)\restriction F$ of all $k$-coarsenings of $F$ is $c$-monochromatic.
\end{cor}
\begin{proof}
Define a mapping $\pi : \E_\infty(\omega \cdot l) \rightarrow \E_k(\omega \cdot l)$ by letting $\pi(E')$, for $E' \in \E_\infty(\omega \cdot l)$, be the equivalence relation obtained by joining each equivalence class with minimal representative $p_n(E'), n \geq k$ to the class with minimal representative $0$. Then $\pi$ is continuous and so $c \circ \pi$ is a finite Souslin measurable colouring of $\E_\infty(\omega \cdot l)$. Thus, by the previous result there is $F \in \E_\infty(\omega \cdot l), F \leq E$ such that $c \circ \pi$ is constant on $\E_\infty(\omega \cdot l)\restriction F$. Let $F^*$ be a coarsening of $F$ such that every equivalence class of $F^*$ contains infinitely many classes of $F$. Then every $F' \in \E_k(\omega \cdot l) \restriction F^*$ can be expressed as $F' = \pi(G)$ for some $G \in \E_\infty(\omega \cdot l) \restriction F$, and so it follows that $\E_k(\omega \cdot l) \restriction F^*$ is $c$-monochromatic.
\end{proof}

\section{An extension of the left-variable Hales-Jewett theorem to infinite alphabets}

In order to extend the results of the previous section to the case where $\P$ is an infinite partition of $\omega$, we will need access to an infinite alphabet $L$ when coding end-extensions. While the natural extension of the left-variable Hales-Jewett theorem to an infinite alphabet is false (see, e.g., \cite[Remark 2.38]{T}), for our purposes we only need a weak version of such a result which we will prove in this section (Theorem \ref{thm5}). To do so, we make use of the theory of idempotent ultrafilters; a brief overview is included below, but we refer the reader to the first few sections of \cite[Chapter 2]{T} for more details.

Suppose $L = \bigcup_{n<\omega} L_n$ is an infinite alphabet given as an increasing chain of finite subsets $L_n$. Let $S = W_L \cup W_{Lv}$ and consider the semigroup $(S, \, ^\frown)$ and its extension $(S^*, \, ^\frown)$, where $S^* = \beta S \setminus S$ and where $\beta S$ is the Stone-\v{C}ech compactification of $(S, \, ^\frown)$. We view $S^*$ as the collection of all non-principal ultrafilters on $S$. Each $\U \in S^*$ corresponds to an \emph{ultrafilter quantifier} in the following way: Given a first-order formula $\varphi(x)$ with a free variable $x$ ranging over elements of $S$, we write $$(\U x) \, \varphi(x) \text{ iff } \{x \in S : \varphi(x)\} \in \U.$$ It is easy to check that ultrafilter quantifiers commute with conjunction and negation of first-order formulas. Using these quantifiers, the extensions of the operations of concatenation and substitution to $S^*$ can be characterized as follows: For $A \subseteq S$ and $\lambda \in L$, $$A \in \U^\frown \V \text{ iff } (\U x)(\V y) \, x^\frown y \in A,$$ $$A \in \U[\lambda] \text{ iff } (\U x) \, x[\lambda] \in A.$$ An ultrafilter $\U \in S^*$ is \emph{idempotent} if $\U ^\frown \U = \U$. Given two idempotent ultrafilters $\U, \V \in S^*$, we write $\U \leq \V$ whenever $$\U ^\frown \V = \V ^\frown \U = \U.$$ Finally, we say that an idempotent ultrafilter $\U \in I \subseteq S^*$ is \emph{minimal} if it is minimal in $I$ with respect to the partial order $\leq$. We will make use of the following standard facts about minimal idempotents in compact semigroups, which we state in our specific context; the reader is referred to \cite[Chapter 2.1]{T} for the general proofs, and to \cite{BBH} or \cite{FK} for more details on the use of idempotent ultrafilters in Ramsey theory.

Below, a subset $I$ of $S^*$ is a \emph{left-ideal} if $I$ is non-empty and $\U ^\frown \V \in I$ for every $\U \in S^*$ and every $\V \in I$. The notions of right-ideal and two-sided ideal are defined similarly.

\begin{lem}
\begin{enumerate}
	\item Every closed subsemigroup of $S^*$ contains a minimal idempotent.
	\item For every two-sided ideal $I \subseteq S^*$ and every idempotent $\U$, there is an idempotent $\V \in I$ such that $\V \leq \U$.
	\item For every minimal idempotent $\V$ and every right-ideal $J \subseteq S^*$, there is an idempotent $\U \in J$ such that $\V^\frown \U = \V$ and $\U ^\frown \V = \U$.
\end{enumerate}
\end{lem}

For a sequence $X = (x_n)_{n<\omega}$ in $W_{Lv}$, let $[X]_L^*$ denote the set of all words in $W_L$ of the form $${x_0[\lambda_0]}^\frown \dots ^\frown{x_k}[\lambda_k]$$ where  $k < \omega$ and $\lambda_i \in L_i$ for each $i \leq k$. We then have the following version of the left-variable Hales-Jewett theorem for infinite alphabets.

\begin{thm}\label{thm5}
For every finite colouring of $W_L$ there is a variable-free word $w_0$ together with an infinite sequence $X = (x_n)_{n<\omega}$ of left-variable words such that the translate ${w_0}^\frown [X]_L^*$ is monochromatic.
\end{thm}
\begin{proof}
We follow the proof of \cite[Theorem 2.37]{T}. Using the previous lemma, let $\W$ be a minimal idempotent belonging to the closed subsemigroup $\{\X \in S^* : W_L \in \X\}$ and let $\V \leq \W$ be a minimal idempotent belonging to the two-sided ideal $\{\X \in S^* : W_{Lv} \in \X\}$. Since $v\,^\frown W_{Lv}$ is a right-ideal of $S$, the collection $$J = \{\X \in S^*: v\,^\frown W_{Lv} \in \X\}$$ is a right-ideal of $S^*$ and so we can choose an idempotent $\U \in J$ such that $$\V^\frown \U = \V \, \text{ and }\, \U^\frown \V = \U.$$ Note that since each substitution mapping $\X \mapsto \X[\lambda]$ $(\lambda \in L)$ is a homomorphism, it follows that each ultrafilter $\V[\lambda]$ is an idempotent of $S^*$ such that $$\V[\lambda] \leq \W[\lambda] = \W,$$ where the equality comes from the fact that the mapping $x \mapsto x[\lambda]$ is the identity on $W_L$. Thus $\V[\lambda] = \W$ by minimality of $\W$; in particular, this implies $$\W^\frown \U[\lambda] = \W \, \text{ and } \, \U[\lambda]^\frown \W = \U[\lambda]$$ for all $\lambda \in L$.

Let $P$ be the colour of the given colouring which belongs to $\W$. By recursion on $n < \omega$, we will build a sequence $(x_n)$ of left-variable words from $W_{Lv}$ and a sequence $(w_n)$ of variable-free words from $W_L$ such that $${w_0}^\frown{x_0}[\lambda_0]^\frown{w_1}^\frown{x_1}[\lambda_1]^\frown \dots ^\frown {x_{k-1}}[\lambda_k]^\frown{w_k} \in P$$ for all $k < \omega$ and $\lambda_i \in L_i, i \leq k$. Then taking $x_n' = {x_n}^\frown w_{n+1}$ for each $n$ will give the required sequence $X$ of left-variable words.

For arbitrary $Q \subseteq W_L, w \in W_L$ and $n < \omega$, define $$Q / w = \{t \in W_L : w^\frown t \in Q\}$$ and $$\partial_n Q = \{w \in Q : Q/w \in \W  \, \text{ and } \, Q/w \in \U[\lambda] \text{ for all $\lambda \in L_n$}\}.$$ Using the equations $\W^\frown \W = \W$ and $\W ^\frown \U[\lambda] = \W$ for all $\lambda \in L$, we see that $\partial_n Q \in \W$ whenever $Q \in \W$ since each $L_n$ is finite.

To start the recursive construction, we use the fact that $P \in \W$ to get $\partial_0 P \in \W$; in particular $\partial_0 P$ is non-empty and so there is $w_0 \in W_L$ such that $$(\U x)(\forall \lambda \in L_0) \, x[\lambda] \in P/w_0.$$ Combining this with the fact that $\U[\lambda]^\frown \W = \U[\lambda]$ for all $\lambda \in L$, we can find a left-variable word $x_0$ from $W_{Lv}$ (since $\U$ concentrates on $v\,^\frown W_{Lv}$) and sets $P_{0,\lambda} \in \W \, \, (\lambda \in L_0)$ such that  $x_0[\lambda] \in P/w_0$ and $x_0[\lambda]^\frown w \in P/w_0$ for all $w \in P_{0,\lambda}$. Let $$P_0 = \bigcap\{P_{0,\lambda} : \lambda \in L_0\}.$$ Since $P_0 \in \W$, we also have $\partial_1 P_0 \in \W$; in particular $\partial_1 P_0$ is non-empty and so there is $w_1 \in W_L$ such that, for all $\lambda \in L_0$, $$P/({w_0} ^\frown x_0[\lambda]^\frown w_1) \in \W \, \text{ and } \, (\forall \mu \in L_1)\, P/({w_0}^\frown x_0[\lambda] ^\frown w_1) \in \U[\mu].$$ Rewriting this fact in terms of $\U$, this means $$(\U x)(\forall \mu \in L_1)(\forall \lambda \in L_0) \, x[\mu] \in P/({w_0}^\frown x_0[\lambda] ^\frown w_1)$$ and so, combining this with the fact that $\U[\lambda]^\frown \W = \U[\lambda]$ for all $\lambda \in L$, we can find a left-variable word $x_1$ in $W_{Lv}$ and sets $P_{1,\mu} \in \W \, \, (\mu \in L_1)$ such that $x_1[\mu] \in P/({w_0}^\frown x_0[\lambda] ^\frown w_1)$ and $x_1[\mu]^\frown w \in P/({w_0}^\frown x_0[\lambda] ^\frown w_1)$ for all $w \in P_{1,\mu}$ and for all $\lambda \in L_0$. Now let $$P_1 = \bigcap \{P_{1,\mu} : \mu \in L_1\}$$ and continue inductively to construct the required sequences $(x_n)$ and $(w_n)$.
\end{proof}

\section{$\P$-alternating equivalence relations for infinite partitions of $\omega$}

In this section we will prove an analogue of Theorem \ref{thm1} in the case where $\omega$ is partitioned into infinitely many pieces. Let $\P = \{P_n : n<\omega\}$ be a partition of $\omega$ into infinite sets such that $\min P_i < \min P_j$ whenever $i < j$. In what follows, we use the same notation as in the case where $\P$ is a finite partition of $\omega$, as it will be clear from context which partition $\P$ (and hence which space of equivalence relations) we are considering. Let $\E^\P_\infty$ denote the set of all equivalence relations $E$ on $\omega$ such that:
\begin{enumerate}
	\item[(a)] $E$ has infinitely many equivalence classes.
	\item[(b)] If $(p_k(E))_{k<\omega}$ is an increasing enumeration of $p(E)$, then $p_k(E) \in P_n$ if and only if there is $m < \omega$ such that $k = 2^n + 2^{n+1}\cdot m - 1$. Equivalently, $p_k(E) \in P_n$ if and only if $n<\omega$ is maximal such that $2^n$ divides $k+1$. (The sequence defined by this condition corresponds to sequence number A007814 in the On-Line Encyclopedia of Integer Sequences.)
\end{enumerate}
Note that condition (b) implies that each $E \in \E^\P_\infty$ has infinitely many $P_n$-classes for each $n < \omega$. (Recall that these are equivalence classes $X$ such that $\min X \in P_n$.) As before, we will say that such an equivalence relation is \emph{$\P$-alternating}. Letting $p_k = p_k(E)$, condition (b) merely states that the first few minimal representatives of $E$ must satisfy $$p_0, p_2, p_4, \dots \in P_0,$$ $$p_1, p_5, p_9, \dots \in P_1,$$ $$p_3, p_{11}, p_{19}, \dots \in P_2,$$ $$\vdots$$ and so on, so that in general we have $p_k \in P_{\sigma(k)}$ where $\sigma$ is the sequence defined by the condition in (b), i.e. $\sigma(k)$ is the largest power of $2$ which divides $k+1$. Note that condition (b) implies that for any $q < \omega$ and any two minimal representatives $p_n, p_m \in P_q$ such that $n > m$, the difference $n - m$ is always a multiple of $2^{q+1}$. Also, any interval $I \subseteq \omega$ of length at least $2^{q+1}$ contains an integer $n$ such that $p_n \in P_q$.

\begin{thm}\label{thm6}
$(\E^\P_\infty, \leq, r)$ is a topological Ramsey space.
\end{thm}

\begin{proof}
It is enough to check axiom A.4. Let $[a, E] \neq \emptyset$ be a basic set, $n = |a|$ and $\O \subseteq (\AE^\P_\infty)_{n+1}$. We can assume $a = r_n(E)$ and we aim to find $F \in [a,E]$ such that $r_{n+1}[a,F] \subseteq \O$ or $r_{n+1}[a,F] \subseteq \O^c$. As before, consider an arbitrary end-extension $b \in r_{n+1}[a,E]$. Then $b$ is an equivalence relation on a set of the form $$p_m(E) = \{0, 1, \dots, p_m(E) - 1\}$$ for some $m > n$, such that $b$ has one more equivalence class with minimal representative $p_n(E)$. Thus $b$ joins the classes of $E$ with minimal representatives among $p_{n+1}(E), \dots, p_{m-1}(E)$ to a class with minimal representative $\leq p_n(E)$. Suppose $p_m(E) \in P_q$ (and note that $q$ only depends on $n$). By our condition for being a member of $\E^\P_\infty$, there is $\lambda <\omega$ such that $m = n + 1 + 2^{q+1}\cdot \lambda$.

Let $(q_i)_{i < \omega}$ be the sequence satisfying $p_{n+1+i} \in P_{q_i}$ for all $i$; in particular $q_0 = q$. Define a sequence $(t_i)_{i<\omega}$ recursively by $$t_0 = 2^{q_0 + 1},$$ $$t_i = t_{i-1} \cdot 2^{q_i + 1} \, \, \, \, \, \, (i > 0)$$ and use this to define a nested sequence of finite alphabets $L_i$ as follows: $$L_0 = (n+1)^{t_0},$$ $$L_i = L_{i-1} \cup (n+1)^{t_i} \, \, \, \, \, \, (i>0).$$ Then let $L = \bigcup_{i<\omega} L_i$. Each $b \in r_{n+1}[a,E]$ of the above form can be coded by a word $w^b \in W_L$ of length $\lambda$, by associating each block of the form $$\left(p_{n+i\cdot t_0 + 1}(E), p_{n+i\cdot t_0+2}(E), \dots, p_{n+(i+1)\cdot t_0}(E)\right), \, \, \,  i < \lambda$$ to the letter $$(k_0, \dots, k_{t_0-1}) \in L_0$$ where $p_{k_j}(E)$ is joined to $p_{n+1+i\cdot t_0+j}(E)$.

Conversely, any word $w \in W_L$ determines an end-extension $b(w) \in r_{n+1}[a,E]$ as follows: First, for each $i < \omega$ let $m_i$ be chosen such that $t_i = t_0 \cdot m_i$. Then each letter $l \in (n+1)^{t_{i(l)}}$ of $w$ determines a word $z_l \in W_{L_0}$ of length $m_{i(l)}$, where the $j^{\mathrm{th}}$ letter of $z_l$ is the sequence $$\left(\pi_{j\cdot t_0}(l), \pi_{j\cdot t_0 + 1}(l), \dots, \pi_{j\cdot t_0 + (t_0 - 1)}(l)\right).$$ By concatenating all such words $z_l$, $w$ determines a word $z \in W_{L_0}$ of length $$\lambda' = \sum_{\text{ $l$ is a letter of $w$} } m_{i(l)}.$$ Now, for each $i < \lambda'$ and $j < t_0$, join the class $p_{n+1+i \cdot t_0 +j}(E)$ to $p_{\pi_j(z(i))}(E)$ and let $b(w)$ be the corresponding equivalence relation on the set $p_m(E)$, where $m = n+1+ t_0 \cdot \lambda'$.

Define a colouring $c : W_L \rightarrow 2$ by setting $c(w) = 0$ if and only if $b(w) \in \O$, and apply Theorem \ref{thm5} to obtain a variable-free word $w_0$ together with a sequence $X = (x_i)_{i<\omega}$ of left-variable words such that the translate ${w_0}^\frown[X]_L^*$ is monochromatic for $c$. In other words, either
\begin{enumerate}
\item $b(w) \in \O$ for every $w \in {w_0}^\frown[X]_L^*$, or
\item $b(w) \not \in \O$ for every $w \in {w_0}^\frown[X]_L^*$.
\end{enumerate}
For each letter $l \in L$, write $\dom(l)$ for the domain of $l$, i.e. for the unique integer $t_j$ such that $l \in (n+1)^{t_j}$. For each occurrence of $v$ in $x_i$ for some $i < \omega$, then we set $\dom(v) = t_i$ (note that the domain depends on $i$). Define a variable-free word $u_0$ together with a sequence of variable words $(y_i)_{i < \omega}$ as follows: First, let $u_0$ be the variable-free word of length $$\sum_{j<|w_0|} \dom(w_0(j))$$ in the alphabet $\{0,\dots, n\}$ obtained by replacing each letter $(k_1, \dots, k_{t_j}) \in (n+1)^{t_j}$ from $w_0$ with the string $$k_1 \dots  k_{t_j}.$$ Then, assuming we have defined $u_0$ and $y_0, \dots, y_{i-1}$, let $y_i$ be the variable word of length $$\sum_{j<|x_i|} \dom(x_i(j))$$ in the alphabet $\{0, \dots, n\}$ which is obtained as follows:
\begin{enumerate}[(i)]
	\item Replace each letter $(k_1, \dots, k_{t_j}) \in (n+1)^{t_j}$ in $x_i$ with the string $$k_1  \dots  k_{t_j}.$$
	\item Replace the left-most variable $v$ in $x_i$ with the string $$0 \dots  0 v  0 \dots 0$$ of length $t_i$, where $v$ occurs in the least place $N < t_i$ such that $q_s$ is equal to $q_i$, where $$s = |u_0| + |y_0| + \dots + |y_{i-1}| + N.$$ Note that such an $N$ exists since $t_i \geq 2^{q_i + 1}$ and so the interval $$[|u_0| + |y_0| + \dots + |y_{i-1}|, |u_0| + |y_0| + \dots + |y_{i-1}|  + t_i)$$ must contain an index $s$ such that $q_s = q_i$ by definition of the sequence $\sigma$.
	\item Assume inductively that we have defined the first $j$ letters of $y_i$, and consider the least occurrence of the variable $v$ in $x_i$ which has not been replaced. Then replace $v$ with the string $$0 \dots  0 v 0 \dots  0$$ of length $t_i$, where $v$ occurs in the least place $N < t_i$ such that $q_s$ is equal to $q_i$, where $$s = |u_0| + |y_0| + \dots + |y_{i-1}| + j + N.$$ As before, such $N$ exists since the interval $$[|u_0| + |y_0| + \dots + |y_{i-1}| + j, |u_0| + |y_0| + \dots + |y_{i-1}| + j + t_i)$$ has length $t_i \geq 2^{q_i + 1}$.
\end{enumerate}
Now form the infinite word $$y = {u_0}^\frown{y_0}^\frown{y_1}^\frown \dots ^\frown{y_k}^\frown \dots$$ out of $u_0$ and $(y_i)_{i<\omega}$. To define $F$, it suffices to say how it acts on $p(E)$. If at place $i$ of $y$ we find a letter $k \in \{0,\dots,n\}$, we let $p_{n+1+i}(E)$ and $p_k(E)$ be $F$-equivalent. If we find a variable at places $$i, i' \in [|u_0| + |y_0| + \dots + |y_{j-1}|, |u_0| + |y_0| + \dots + |y_{j-1}| + |y_j|)$$ (where we set $y_{-1} = \emptyset$) then we let $p_{n+1+i}(E)$ and $p_{n+1+i'}(E)$ be $F$-equivalent, with no other connections. Then, by our placement of each of the variables, it follows that $F$ is $\P$-alternating. As in the proof of Theorem \ref{thm1}, $F$ satisfies the conclusion of A.4.
\end{proof}

\begin{cor}\label{cor6}
Suppose $c$ is a finite Souslin measurable colouring of $\E^\P_\infty$. Then for every $E \in \E^\P_\infty$ there is $F \leq E$ such that the family $\E^\P_\infty\restriction F$ of all coarsenings of $F$ is $c$-monochromatic.
\end{cor}

We can also impose conditions on the elements of each class of an equivalence relation: Given a sequence $\I = (I_n)_{n<\omega}$ of subsets of $\omega$ such that $n \in I_n$ for all $n < \omega$, let $\E^{\P, \I}_\infty$ be the set of all $\P$-alternating equivalence relations on $\omega$ which also satisfy:
\begin{enumerate}
	\item[(c)$_\I$] If $X$ is a $P_n$-class, then $X \subseteq \bigcup \{ P_m : m \in I_n\}$.
\end{enumerate}
Modifying the proof of Theorem \ref{thm6} in the natural way (as in the proof of Theorem \ref{thm2}) yields the following results.

\begin{thm}\label{thm7}
$(\E^{\P, \I}_\infty, \leq, r)$ is a topological Ramsey space.
\end{thm}

\begin{cor}\label{cor7}
Suppose $c$ is a finite Souslin measurable colouring of $\E^{\P, \I}_\infty$. Then for every $E \in \E^{\P, \I}_\infty$ there is $F \leq E$ such that the family $\E^{\P, \I}_\infty\restriction F$ of all coarsenings of $F$ is $c$-monochromatic.
\end{cor}

Our next goal is to apply Corollary \ref{cor7} to obtain dual Ramsey results for countable limit ordinals above or equal to $\omega^2$. Unless otherwise specified, for the rest of this section we fix a countable limit ordinal $\alpha \geq \omega^2$. Using left division of ordinals, we find a countable ordinal $\beta \geq \omega$ such that $\alpha = \omega \cdot \beta$. (The fact that the remainder is $0$ follows from the assumption that $\alpha$ is a limit.) Fix a bijection $f : \omega \rightarrow \beta$ such that $f(0) = 0$. Let $\P$ be any partition of $\omega$ as above, and identify each $P_n$ with the $f(n)^{\mathrm{th}}$ copy of $\omega$ in $\alpha$, i.e. the set $\omega \times \{f(n)\}$. We remark that unlike the case for the ordinals $\omega \cdot l$, in general there is no natural way to ``order'' the copies of $\omega$ in $\alpha$ via a bijection $f$. However, for an ordinal $\alpha < \varepsilon_0$, there is a ``default'' choice for $f$ given by first constructing a canonical bijection between $\omega$ and $\varepsilon_0$ using the Cantor normal form and then restricting to $\alpha$. Since we will not make use of this canonical bijection, we leave the details to the interested reader.

As before, we work with equivalence relations on $\alpha$ which arise from rigid surjections $\alpha \rightarrow \alpha$, and so it is necessary to prescribe conditions on the set of minimal representatives of such an equivalence relation. In this case, we ask that the representatives behave according to the sequence used to define the space $\E^\P_\infty$ introduced above. To formalize this requirement, let $\sigma$ be the sequence used in the part (b) of the definition of the space $\E^\P_\infty$ and consider the bijection $\omega \rightarrow \alpha$ defined by $$n \mapsto \big(| \{m < n : \sigma(m) = \sigma(n)\} |, f(\sigma(n))\big).$$ Then $\alpha$ inherits a linear ordering $\preceq_f$ via this bijection in such a way that $(\alpha, \preceq_f)$ has order type $\omega$. If we enumerate the elements of $\alpha$ as $(\gamma_n)_{n < \omega}$ in $\preceq_f$-increasing order, then $$\gamma_0, \gamma_2, \gamma_4, \dots \in \omega \times \{f(0)\},$$ $$\gamma_1, \gamma_5, \gamma_9, \dots \in \omega \times \{f(1)\},$$ $$\gamma_3, \gamma_{11}, \gamma_{19}, \dots \in \omega \times \{f(2)\},$$ $$\vdots$$ etc., and in general $\gamma_n \in \omega \times \{f(\sigma(n))\}$, so that $(\alpha, \preceq_f)$ is ordered according to the sequence $\sigma$. Let $\E^f_\infty(\alpha)$ denote the set of all equivalence relations $E$ on $\alpha$ such that:
	\begin{enumerate}[(a)]
		\item If $(p_k)_{k<\omega}$ is a $\preceq_f$-increasing enumeration of $p(E)$, then $p_k \in \omega \times \{f(\sigma(k))\}$ for each $k$.
		\item  If $(q_k)_{k<\omega}$ is a $\preceq_f$-increasing enumeration of the set of $\preceq_f$-minimal representatives of $E$, then $q_k = p_k$ for all $k$.
	\end{enumerate}
For each $n$, let $$I_n = \{m < \omega : f(m) \geq f(n)\}$$ and $\I = (I_n)_{n<\omega}$. As in the case for $\E_\infty(\omega \cdot l)$, we can define bijections $$\varphi : \omega \rightarrow \alpha \text{ and } \Phi : \E^{\P, \I}_\infty \rightarrow \E^f_\infty(\alpha)$$ which allow us to transfer equivalence relations on $\omega$ to equivalence relations on $\alpha$. Note that condition (c)$_\I$ ensures that $\phi(p(E)) = p(\Phi(E))$ for each $E \in \E^{\P, \I}_\infty$.

Below, the notion of Souslin measurability can refer to either the topology on $\E^f_\infty(\alpha)$ inherited from $\E^{\P, \I}_\infty$ via $\Phi$ (where the latter set is equipped with the Ellentuck topology or the metrizable topology), or the topology induced from $2^{\alpha^2}$ when each equivalence relation in $\E^f_\infty(\alpha)$ is identified with a subset of $\alpha^2$. The following results extend the dual Silver and dual Ramsey theorems to the ordinal $\alpha$. The proofs are exactly as in the previous section.

\begin{cor}\label{maincor2}
Suppose $c$ is a finite Souslin measurable colouring of $\E^f_\infty(\alpha)$. Then for every $E \in \E^f_\infty(\alpha)$ there is $F \leq E$ such that the family $\E^f_\infty(\alpha)\restriction F$ of all coarsenings of $F$ is $c$-monochromatic.
\end{cor}

To state the corresponding dual Ramsey theorem for $\alpha$, fix $k < \omega$ and a bijection $f$ as above. Let $\E^f_k(\alpha)$ denote the set of all equivalence relations $E$ on $\alpha$ with exactly $k$ equivalence classes such that if $(p_i)_{i<k}$ is a $\preceq_f$-increasing enumeration of $p(E)$, then $p_i \in \omega \times \{ f(\sigma(i)) \}$ for all $i < k$, and if $(q_i)_{i<k}$ is a $\preceq_f$-increasing enumeration of the set of $\preceq_f$-minimal representatives of $E$, then $q_i = p_i$ for all $i<k$. Equip $\E^f_k(\alpha)$ with the topology inherited from $2^{\alpha^2}$ by identifying each equivalence relation in $\E^f_k(\alpha)$ with a subset of $\alpha^2$. 

\begin{cor}
Suppose $c$ is a finite Souslin measurable colouring of $\E^f_k(\alpha)$. Then for every $E \in \E^f_\infty(\alpha)$ there is $F \leq E$ such that the family $\E^f_k(\alpha)\restriction F$ of all $k$-coarsenings of $F$ is $c$-monochromatic.
\end{cor}

At this point, the reader will note that we have only been concerned with the establishment of dual Ramsey theorems for countable limit ordinals. However, it is not difficult to verify that the methods of this paper can also be used to obtain similar results for a countable successor ordinal $\alpha$ if one is willing to work with surjections $f : \alpha \rightarrow \alpha$ which are rigid on the ``limit'' part of $\alpha$; that is, writing $\alpha = \beta + n$ for a limit ordinal $\beta$ and $n<\omega$, we require $f\restriction \beta : \beta \rightarrow \beta$ to be a rigid surjection. On the other hand, any rigid surjection $f : \beta + n \rightarrow \beta + n$ is necessarily equal to the identity when restricted to the interval $[\beta, \beta + n)$ and so the corresponding rigid surjection version of the dual Ramsey theorem follows from the dual Ramsey theorem for $\beta$.

We conclude by noting that the results of this paper can be seen more generally as part of a program which studies the linear orderings for which a version of the dual Ramsey theorem holds. (It is worth mentioning here that some results along these lines are known regarding the dual Ramsey behaviour of trees; see, for instance, \cite{Solecki, TT}.) While the classical versions (corresponding to the infinite Ramsey theorem and the Galvin-Prikry theorem) of this problem have been well-studied, much remains unknown in the dual context, where one works with an appropriate notion of rigid surjection in place of the order-preserving injections which appear in the formulation of classical Ramsey-theoretic results on linear orderings. In this paper we have seen that this is possible for countable ordinals, but it is unclear if the methods can be extended to study the dual Ramsey behaviour of a wider class of linear orderings. Hence, the following general problem remains open:

\begin{prob}
Find and characterize linear orderings for which a version of the dual Ramsey theorem holds.
\end{prob}

\bibliographystyle{plain}
\bibliography{bibliography}

\end{document}